\documentclass{amsart} 

\usepackage[breaklinks=true, pdfauthor={Alexander Engel, Martin Weilandt}, pdftitle={Isospectral Alexandrov Spaces}, pdfkeywords={spectral geometry, Laplace operator, isospectrality, Alexandrov spaces}, pdfsubject={differential geometry}]{hyperref}

\usepackage[latin1]{inputenc} % Umlaute direkt hier im Source
\usepackage[T1]{fontenc}      % Schicke Umlaute
\usepackage{amsmath}  % Mathe-Improvements
\usepackage{amsfonts} % Mathe-Fonts
\usepackage{amssymb}  % Mathe-Symbole
\usepackage{amsthm}
\usepackage{enumerate}
\usepackage{fancyhdr} % Schicke Header
\usepackage{comment}
\usepackage{color}
  
\newcommand{\N}{\mathbb{N}}
\newcommand{\Z}{\mathbb{Z}}
\newcommand{\Q}{\mathbb{Q}}
\newcommand{\R}{\mathbb{R}}
\newcommand{\C}{\mathbb{C}}
\newcommand{\CP}{\mathbb{CP}}

\newcommand{\ol}[1]{\overline{#1}}

\DeclareMathOperator{\tr}{tr}

\DeclareMathOperator{\grad}{grad}

\DeclareMathOperator{\Isom}{Isom}

\DeclareMathOperator{\dvol}{dvol}
\DeclareMathOperator{\Id}{Id}
\DeclareMathOperator{\Aut}{Aut}
\DeclareMathOperator{\Diffeo}{Diffeo}
\DeclareMathOperator{\conj}{conj}

\newcommand{\co}{\colon\thinspace}

\theoremstyle{plain}
\newtheorem{theorem}{Theorem}[section]
\newtheorem*{theorem*}{Theorem}  % Theorem ohne Nummer
\newtheorem{lemma}[theorem]{Lemma}

\newtheorem{proposition}[theorem]{Proposition}

\theoremstyle{definition}
\newtheorem{definition}[theorem]{Definition}

\newtheorem{notations}[theorem]{Notations and Remarks}

\theoremstyle{remark}
\newtheorem{remark}[theorem]{Remark}
\newtheorem*{remark*}{Remark}

\newcommand*{\reg}{\text{reg}} % Notation fuer regulaere Orbifoldpunkte
\newcommand*{\sing}{\text{sing}}
\newcommand*{\princ}{\text{princ}} % principal stratum

\newcommand{\q}[1]{[#1]} % Schreibweise fuer Punkte in G\M

\newcommand{\frakt}{\mathfrak{t}}
\newcommand{\frakh}{\mathfrak{h}}
\newcommand{\frakg}{\mathfrak{g}}
\newcommand{\fraksu}{\mathfrak{su}}
\newcommand{\fraku}{\mathfrak{u}}
\newcommand{\fraki}{\mathfrak{i}}

\newcommand{\Ad}{\operatorname{Ad}}

\newcommand{\cali}{\mathcal{I}}
\newcommand{\calj}{\mathcal{J}}
\newcommand{\calZ}{\mathcal{Z}}

\begin{document}

\title{Isospectral Alexandrov Spaces}
\author{Alexander Engel}
\address{Institut f\"ur Mathematik, Universit\"at Augsburg, 86135 Augsburg, Germany}
\email{alexander.engel@math.uni-augsburg.de}
\author{Martin Weilandt}
\address{Departamento de Matem\'atica, Universidade Federal de Santa Catarina, 88040-900 Florian\'opolis-SC, Brazil}
\email{martin.weilandt@ufsc.br}

\maketitle

\begin{abstract}
We construct the first non-trivial examples of compact non-isometric Alexandrov spaces which are isospectral with respect to the Laplacian and not isometric to Riemannian orbifolds. This construction generalizes independent earlier results by the authors based on Sch\"uth's version of the torus method.
\end{abstract}

\numberwithin{equation}{section}

%\setcounter{tocdepth}{1}
%\tableofcontents

%%%%%%%%%%%%%%%%%%%%%%%%%%%%%%% 1. Introduction
\section{Introduction}
Spectral geometry is the study of the connection between the geometry of a space and the spectrum of the associated Laplacian. In the case of compact Riemannian manifolds the spectrum determines geometric properties like dimension, volume and certain curvature integrals (\cite{MR0282313}). Besides, various constructions have been given to find properties which are not determined by the spectrum (see \cite{MR1736857} for an overview). Many of those results could later be generalized to compact Riemannian orbifolds (see \cite{MR2985309} and the references therein). More generally, one can consider the Laplacian on compact Alexandrov spaces (by which we always refer to Alexandrov spaces with curvature bounded from below). The corresponding spectrum is also given by a sequence $0=\lambda_0\le\lambda_1\le\lambda_2\le\ldots\nearrow\infty$ of eigenvalues with finite multiplicities (\cite{MR1865418}) and we call two compact Alexandrov spaces isospectral if these sequences coincide. Although it is not known which geometric properties are determined by the spectrum in this case, one can check if constructions of isospectral manifolds (or, more generally, orbifolds) carry over to Alexandrov spaces. In this paper we observe that this is the case for the torus method from \cite{MR1895349}, which can be used to construct isospectral Alexandrov spaces as quotients of isospectral manifolds. The same idea has already been used in \cite{engel_bachelor} and \cite{weilandt10a} to construct isospectral manifolds or orbifolds out of known families of isospectral manifolds. Our main result is the following.

\begin{theorem*}
For every $n\ge 4$ there are continuous families of isospectral $2n$-di\-men\-sio\-nal Alexandrov spaces which are pairwise non-isometric and none of which is isometric to a Riemannian orbifold. Furthermore, they can be chosen not to be products of non-trivial Alexandrov spaces.
\end{theorem*}

\begin{remark*}
We also give such families in every dimension $4n-1$, $n \ge 4$, but could not establish sufficient criteria for the non-isometry of these more complicated examples.
\end{remark*}

%%%%%%%%%%%%%%%%%%%%%%%%%%%%%%% 2. Alexandrov Spaces
\section{Alexandrov Spaces}
Although we will not actually work with the general definition of an Alexandrov space but mainly with the construction given in Proposition \ref{proposition:quotient}, we include it here for completeness. We follow the definition from \cite{MR1865418}, also compare \cite{MR1835418,MR1185284}.

Let $(X, d)$ be a metric space and $\gamma\co [a, b] \to X$ a continuous path. Then the \emph{length} $L(\gamma)$ of $\gamma$ is defined as $L(\gamma) := \sup_{z_0, \ldots, z_N} \sum_{i=1}^N d(\gamma(z_{i-1}), \gamma(z_i))$, where the supremum is taken over all partitions $a= z_0 \le z_1 \le \ldots \le z_N = b$ of the interval $[a, b]$. The metric space $(X, d)$ is called \emph{intrinsic} if the distance $d(x,y)$ between two points $x, y \in X$ is the infimum of the lengths of all continuous paths from $x$ to $y$.

It is known (see, e.g., \cite[Theorem 2.5.23]{MR1835418}) that in a complete, locally compact, intrinsic metric space $(X, d)$ every two points $x, y \in X$ with $d(x, y) < \infty$ are connected by a \emph{shortest path}, i.e., there exists a continuous path $\gamma\co [a, b] \to X$ with $\gamma(a) = x$, $\gamma(b) = y$ and $L(\gamma) = d(x, y)$. Note that a shortest path is not necessarily unique.

By a \emph{triangle} $\Delta x_1x_2x_3$ in a complete, locally compact, intrinsic metric space we mean a collection of three points $x_1,x_2,x_3$ connected by shortest paths (\emph{sides}) $[x_1x_2],[x_2x_3],[x_1x_3]$.

\begin{definition}
  \label{def:as}
  Let $K\in\R$. A complete, locally compact, intrinsic metric space $(X,d)$ of finite Hausdorff dimension is called an \emph{Alexandrov space of curvature $\ge K$} if in some neighbourhood $U$ of each point in $X$ the following holds: For every triangle $\Delta x_1x_2x_3$ in $U$ there is a (comparison) triangle $\Delta \bar{x}_1\bar{x}_2\bar{x}_3$ in the simply connected 2-dimensional space form $(M_K,\bar{d})$ of constant curvature $K$ with the following properties:
  \begin{itemize}
  \item The sides of $\Delta\bar{x}_1\bar{x}_2\bar{x}_3$ have the same lengths as the corresponding sides of $\Delta x_1x_2x_3$.
  \item If $y\in [x_1x_3]$ and $\bar{y}$ denotes the point on the side $[\bar{x}_1\bar{x}_3]$ with $\bar{d}(\bar{x}_1,\bar{y})=d(x_1,y)$, then $d(x_2,y)\ge \bar{d}(\bar{x}_2,\bar{y})$.
  \end{itemize}
\end{definition}

Now fix an Alexandrov space $X$ of Hausdorff dimension $n$ (which is automatically a non-negative integer). A point $x$ in $X$ is called \emph{regular} if for $\varepsilon\to 0$ the open balls $\frac{1}{\varepsilon}B_\varepsilon(x)\subset X$ converge to the unit ball in the Euclidean space $\R^n$ with respect to the Gromov-Hausdorff metric. The set $X^\reg$ of regular points is known to be dense in $X$.

Unless otherwise stated, we assume that $X^\reg$ is open in $X$. Then it is known (see \cite{kuwae07} and the references therein):

\begin{proposition}\label{proposition:alex-mf}
\begin{enumerate}[(i)]
\item $X^\reg$ has a natural $n$-dimensional $C^\infty$-manifold structure.
\item \label{item:alex-metric} There is a unique continuous Riemannian metric $h$ on $X^\reg$ such that the induced distance function $d_h$ coincides with the original metric $d$ on $X^\reg$.
\item If $X$ carries the structure of a manifold with (smooth) Riemannian metric, the $C^\infty$-structure and the continuous Riemannian metric $h$ mentioned above coincide with the given $C^\infty$-structure on $X=X^\reg$ and the given smooth Riemannian metric.
\end{enumerate}
\end{proposition}

Apart from Riemannian manifolds (or, more generally, orbifolds) with sectional curvature bounded below, another way to construct Alexandrov spaces is the following proposition.

\begin{proposition}
  \label{proposition:quotient}
  Let $K\in\R$, let $(M,g)$ be a compact Riemannian manifold with sectional curvature $\ge K$ and let $G$ be a compact Lie group acting isometrically on $(M,g)$. Then the quotient space $(G\backslash M,d_g)$ (with $d_g(\q{x},\q{y})=\inf_{a\in G}d_g(x,ay)$) is an Alexandrov space of curvature $\ge K$.
\end{proposition}

Note that, in the setting above, the compactness of $G$ implies that a shortest path between two arbitrary points in $(G\backslash M,d_g)$ can be found by just pushing down an appropriate shortest path in $(M,g)$. For a proof of the curvature bound in the quotient, just note that, by Toponogov's theorem, $(M,g)$ is an Alexandrov space of curvature $\ge K$ and apply \cite[Proposition 10.2.4]{MR1835418} (also see \cite{MR1185284}). Moreover, note that Proposition \ref{proposition:alex-mf} directly implies that $(G\backslash M,d_g)$ has Hausdorff dimension $\dim M-\dim G$. Moreover, note that if the action is almost free (i.e., all stabilizers $G_x$, $x\in M$, are finite), then the quotient $(G\backslash M,d_g)$ above is isometric to a Riemannian orbifold (\cite{MR932463}).

%%%%%%%%%%%%%%%%%%%%%%%%%%%%%%% 3. The Laplacian
\section{The Laplacian}
The Laplacian on Alexandrov spaces was introduced in \cite{MR1865418}: Assume that $X$ is a compact Alexandrov space. The Sobolev space $H^1(X,\R)\subset L^2(X,\R)$ is by definition given by all measurable real-valued functions on $X$ whose restrictions to $X^\reg$ lie in $H^1((X^\reg,h),\R)$ (with $h$ the continuous Riemannian metric from Proposition \ref{proposition:alex-mf}(\ref{item:alex-metric}), which we also denote by $\langle,\rangle$). The scalar product in $H^1(X,\R)$ is then given by $(u,v)_1:=\int_X uv+\int_{X^\reg}\langle \nabla u,\nabla v\rangle$ for $u,v\in H^1(X,\R)$, where $\nabla$ stands for the weak derivative as $L^2$-vector field. For $u,v\in H^1(X,\R)$ set
\[\mathcal{E}(u,v):=\int_{X^\reg} \langle\nabla u, \nabla v\rangle.\]
By \cite{MR1865418} there is a maximal self-adjoint operator $\Delta\co\mathcal{D}(\Delta)\to L^2(X,\R)$ such that $\mathcal{D}(\Delta)\subset H^1(X,\R)$ and $\mathcal{E}(u,v)=\int_X u\Delta v$ for $u\in H^1(X,\R),v\in\mathcal{D}(\Delta)$. The last equation implies that for $u,v\in\mathcal{D}(\Delta)$:

\begin{equation}
  \label{equation:green}
  \int_X u\Delta v=\mathcal{E}(u,v)=\mathcal{E}(v,u)=\int_X v\Delta u.
\end{equation}
Moreover, by \cite{MR1865418} there is an orthonormal basis $(\phi_k)_{k\ge 0}$ of $L^2(X,\R)$ and a sequence $0\le\lambda_0\le\lambda_1\le\lambda_2\le\ldots\nearrow\infty$ such that each $\phi_k$ is locally H\"older continuous, lies in $\mathcal{D}(\Delta)$ and $\Delta\phi_k=\lambda_k\phi_k$.

Since the torus method is based on representation theory on complex Hilbert spaces, we will need the $\C$-linear extensions of $\nabla$ and $\Delta$ given by the complex gradient $\nabla^\C$ on $H^1(X):=H^1(X,\R)\otimes\C$ and the complex Laplacian $\Delta^\C$ on $\mathcal{D}(\Delta^\C):=\mathcal{D}(\Delta)\otimes\C$. We also extend $\langle,\rangle$ to a Hermitian form $\langle,\rangle_\C$ on complex-valued $L^2$-vector fields. \eqref{equation:green} then implies $\int_X\langle\nabla^\C u,\nabla^\C v\rangle_\C=\int_Xu\overline{\Delta^\C v}$ for $u,v\in\mathcal{D}(\Delta^\C)$. Using the existence of the orthonormal basis $(\phi_k)_{k\ge 0}$ of $L^2(X,\C)$, we obtain the following variational characterization of eigenvalues:
\begin{equation}
\label{eq:variational_characterization_eigenvalues}
\lambda_k=\inf_{U\in L_k}\sup_{f\in U\setminus\{0\}}\frac{\int \langle \nabla^\C f,\nabla^\C f\rangle_\C}{\int |f|^2}
\end{equation}
with $k\ge 0$ and $L_k$ the set of $k$-dimensional subspaces of $H^1(X)$ (compare the manifold setting in \cite{MR861271}).

We will call two compact Alexandrov spaces $X_1$, $X_2$ \emph{isospectral} if the corresponding sequences $(\lambda_k(X_1))_{k\ge0},(\lambda_k(X_2))_{k\ge0}$ of eigenvalues (repeated according to their multiplicities) coincide.

%%%%%%%%%%%%%%%%%%%%%%%%%%%%%%% 4. The Torus Method on Alexandrov Spaces
\section{The Torus Method on Alexandrov Spaces}
\label{section:isometrics}
Let $G$ be a compact, connected Lie group acting isometrically and effectively on a compact, connected, smooth Riemannian manifold $(M,g)$. The Riemannian metric $g$ induces a distance function $d_g$ on $M$, which in turn induces a distance function on $G\backslash M$. The latter will be denoted by $d_g$, too (compare Proposition \ref{proposition:quotient}). The union $M_{\text{princ}}$ of principal $G$-orbits is open and dense in $M$ and $G\backslash M_{\text{princ}}$ is a connected manifold (\cite{alexandrino09}). Note that if $G$ is abelian (as in our examples) then $M_{\text{princ}}$ is simply the set of points in $M$ with trivial stabilizer. Also note that $g$ induces a Riemannian submersion metric $h$ on the quotient $G\backslash M_{\text{princ}}$, which coincides with the continuous Riemannian metric induced on $G\backslash M_{\text{princ}}$ when considered as the regular part of the Alexandrov space $(G\backslash M,d_g)$ (Proposition \ref{proposition:alex-mf}).

Now let $T$ be a non-trivial torus (i.e., a compact, connected, abelian Lie group) acting effectively and isometrically on $(M,g)$ such that the $G$- and $T$-actions commute and such that the induced $T$-action on $G\backslash M$ is also effective. It follows that $M_{\text{princ}}$ is $T$-invariant and this induced $T$-action is isometric on $(G\backslash M_{\text{princ}}, h)$. By $\widehat{M}$ denote the set of all $x \in M_{\text{princ}}$ for which $\q{x}\in G\backslash M_{\text{princ}}$ has trivial $T$-stabilizer. Since $G\backslash \widehat{M}$ is open and dense in $G\backslash M_{\text{princ}}$ and the projection $\xi\co M_\princ\to G\backslash M_\princ$ is a submersion, $\widehat{M}$ is open and dense in $M_{\text{princ}}$ and hence also open and dense in $M$.

In the theorem below and later on we will also use the following notation: Let ${\mathfrak t} =T_eT$ denote the Lie algebra of $T$. Setting ${\mathcal L}=\ker(\exp\co{\mathfrak t}\to T)$, we observe that $\exp$ induces an isomorphism from ${\mathfrak t}/{\mathcal L}$ to $T$. Let ${\mathcal L}^*:=\{\phi\in {\mathfrak t}^*;~\phi(X)\in\Z~\forall X\in{\mathcal L}\}$ denote the dual lattice.
If $W$ is a subtorus of $T$ and $h$ the Riemannian submersion metric on $G\backslash \widehat{M}$ induced from $(\widehat{M}, g)$, we write $h^W$ to denote the induced Riemannian submersion metric on $W\backslash (G\backslash \widehat{M})$. Moreover, given some Riemannian metric $g$ on a manifold, we write $\dvol_g$ for the Riemannian density associated with $g$.

The original version of the following theorem (without $G$-action) was given in \cite{MR1895349}, also compare \cite{weilandt10a} for the case of an almost free $G$-action:

\begin{theorem}
  \label{theorem:torusaction}
  Let $G$ and $T$ be two compact, connected Lie groups acting effectively and isometrically on two compact Riemannian manifolds $(M,g)$ and $(M^\prime,g^\prime)$. Let $h$ and $h^\prime$ denote the induced Riemannian metrics on $G\backslash M_\princ$ and $G\backslash M^\prime_\princ$, respectively, and assume that $G$- and $T$-actions on each manifold commute and the induced isometric $T$-actions on $(G\backslash M_\princ,h)$ and $(G\backslash M^\prime_\princ,h^\prime)$ are effective. Moreover, assume that for every subtorus $W$ of $T$ of codimension 1 there is a $G$- and $T$-equivariant diffeomorphism $E_W\co M\to M^\prime$ which satisfies $E_W^*\dvol_{g^\prime}=\dvol_g$ and induces an isometry between the manifolds $(W\backslash (G\backslash \widehat{M}),h^W)$ and $(W\backslash (G\backslash \widehat{M}^\prime),{h^\prime}^W)$. Then the Alexandrov spaces $(G\backslash M,d_g)$ and $(G\backslash M^\prime,d_{g^\prime})$ are isospectral.
\end{theorem}

\begin{proof}
We simply imitate the proof for the case of a trivial $G$ in \cite[Theorem 1.4]{MR1895349}: Consider the (complex) Sobolev spaces $H:=H^1(G\backslash M,d_g)$ and $H^\prime:=H^1(G\backslash M^\prime,d_{g^\prime})$. To construct an isometry $H^\prime\to H$ preserving $L^2$-norms, consider the unitary representation of $T$ on $H$ given by $(zf)(\q{x}):= f(z\q{x})$ for $z\in T$, $f\in H$, $\q{x}\in G\backslash M$. Then the $T$-module $H$ decomposes into the Hilbert space direct sum
\[H=\bigoplus_{\mu\in{\mathcal L}^*}H_\mu\]
of $T$-modules $H_\mu=\{f\in H;~ [Z] f=e^{2\pi i \mu(Z)}f~\forall Z\in {\mathfrak t}\}$. Note that $H_0$ is just the space of $T$-invariant functions in $H$.

For each subtorus $W$ of $T$ of codimension 1 set
\[S_W:=\bigoplus_{\substack{\mu\in{\mathcal L}^*\setminus\{0\}\\ T_eW=\ker\mu}}H_\mu\]
and denote the (Hilbert) sum over all these subtori by $\bigoplus_W$. We obtain
\[H=H_0\oplus \bigoplus_{\mu\in{\mathcal L}^*\setminus\{0\}}H_\mu=H_0\oplus\bigoplus_W S_W.\]
Moreover, set 
\[H_W:=H_0\oplus S_W =\bigoplus_{\substack{\mu\in{\mathcal L}^*\\ T_eW\subset \ker\mu}}H_\mu\]
and note that $H_W$ consists precisely of the $W$-invariant functions in $H$.

Now use the analogous notation $H^\prime_\mu, S^\prime_W, H^\prime_W$ for the corresponding subspaces of $H^\prime$. Fix a subtorus $W$ of $T$ of codimension 1 and let $E_W\co M\to M^\prime$ be the corresponding diffeomorphism from the assumption. $E_W$ induces a $T$-equivariant diffeomorphism $F_W\co G\backslash M_\princ\to G\backslash M_\princ^\prime$, whose pull-back $F_W^*$ maps $H_0^\prime$ to $H_0$ and $H_W^\prime$ to $H_W$.
We will show that $F_W^*\co H_W^\prime\to H_W$ is a Hilbert space isometry preserving the $L^2$-norm. It obviously preserves the $L^2$-norm because $F_W^*\dvol_{h^\prime}=\dvol_h$ on $G\backslash M_{\text{princ}}$.

Let $\psi\in C^\infty(G\backslash \widehat{M}^\prime)$ be invariant under $W$, let $\q{y}\in G\backslash \widehat{M}^\prime$ and set $\phi=\psi\circ F_W$, $\q{x}:=F_W^{-1}(\q{y})\in G\backslash \widehat{M}$. Since $\grad\phi$ and $\grad\psi$ (where $\grad$ denotes the smooth gradient, which coincides almost everywhere with $\nabla$) are $W$-horizontal vector fields on $G\backslash \widehat{M}$, $G\backslash \widehat{M}^\prime$, respectively, and the map $\overline{F}_W\co (W\backslash (G\backslash \widehat{M}),h^W)\to(W\backslash (G\backslash \widehat{M}^\prime),{h^\prime}^W)$ induced by $F_W$ is an isometry, we obtain $\|\grad\phi(\q{x})\|_h=\|\grad\psi(\q{y})\|_{h^\prime}$.
Since $G\backslash \widehat{M}$ is dense in $(G\backslash M)^\reg=G\backslash M_{\text{princ}}$ and $G\backslash \widehat{M^\prime}$ is dense in $(G\backslash M^\prime)^\reg=G\backslash M^\prime_\princ$, this implies that $F_W^*\co H_W^\prime\to H_W$ is a Hilbert space isometry with respect to the $H^1$-product. Since the map $F_W^*\co H_W^\prime\to H_W$ is a Hilbert space isometry preserving $L^2$-norms, so is its restriction ${F_W^*}|_{S_W^\prime}\co S_W^\prime\to S_W$. 

But these maps for all subtori $W\subset T$ of codimension 1 give an isometry from $\bigoplus_W S_W^\prime$ to $\bigoplus_W S_W$ preserving $L^2$-norms. Choosing an isometry $H_0^\prime\to H_0$ given by an arbitrary $F_W^*$, we obtain an $L^2$-norm-preserving isometry $H^\prime\to H$. Isospectrality of $(G\backslash M,d_g)$ and $(G\backslash M^\prime,d_{g^\prime})$ now follows from \eqref{eq:variational_characterization_eigenvalues}.
\end{proof}

We now fix a Riemannian manifold $(M,g_0)$, a Lie group $G$ and a non-trivial torus $T$ as above. If $Z\in{\mathfrak t}$, we write $Z^\#_x:=\tfrac{d}{dt}|_{t=0}\exp(tZ)x$ for the fundamental vector field on $M$ induced by $Z$. We also write $Z^*$ for the corresponding fundamental vector field on $G\backslash M_{\text{princ}}$. We will need the following notation based on ideas in \cite[1.5]{MR1895349}:

\begin{notations}\label{notations:one_forms}

\begin{enumerate}[(i)]
 \item \label{enumremark:admissible} A ${\mathfrak t}$-valued $1$-form on $M$ is called \emph{admissible} if it is $G$- and $T$-invariant and horizontal with respect to both $G$ and $T$. (The latter condition means that the form vanishes on the fundamental vector fields induced by the $G$- or the $T$-action.)
 \item Now fix an admissible ${\mathfrak t}$-valued $1$-form $\kappa$ on $M$ and denote by $g_\kappa$ the $G$- and $T$-invariant Riemannian metric on $M$ given by
\[g_\kappa(X,Y):={g_0}(X+\kappa(X)^\#,Y+\kappa(Y)^\#)\]
for $X,Y\in{\mathcal V}(M)$. It has been noted in \cite{MR1895349} that $\dvol_{g_{\kappa}}=\dvol_{g_{0}}$.

\item \label{notations:one_forms-iii} Since $\kappa$ is admissible, it induces a $\mathfrak{t}$-valued 1-form $\lambda$ on $G\backslash M_\princ$, which is $T$-invariant and $T$-horizontal. Now let $h_0$ denote the Riemannian submersion metric on $G\backslash M_\princ$ induced by $g_0$. Recalling that $\xi\co M_\princ\to G\backslash M_\princ$ denotes the quotient map and letting tildes denote horizontal lifts to $M_\princ$ with respect to $g_\kappa$, we can determine the Riemannian submersion metric on $G\backslash M_\princ$ induced by $g_\kappa$ as
\begin{align*}
  (X,Y)\mapsto g_\kappa(\widetilde{X},\widetilde{Y})&=g_0(\widetilde{X}+\kappa(\widetilde{X})^\#,\widetilde{Y}+\kappa(\widetilde{Y})^\#)\\
  &=h_0(\xi_\ast \widetilde{X}+\xi_\ast(\kappa(\widetilde{X})^\#),\xi_\ast \widetilde{Y}+\xi_\ast(\kappa(\widetilde{Y})^\#))\\
  &=h_0(X+\lambda(X)^\ast,Y+\lambda(Y)^\ast)=:h_\lambda(X,Y).
\end{align*}
(To get from the first to the second line, we used that, since $\widetilde{X}$ is horizontal with respect to $g_\kappa$, the field $\widetilde{X}+\kappa(\widetilde{X})^\#$ is horizontal with respect to $g_0$. Analogously for $Y$.)

\item Since $\lambda$ and $h_0$ on $G\backslash M_{\text{princ}}$ are $T$-invariant, so is the metric $h_\lambda$ on $G\backslash M_{\text{princ}}$.

\item Note that for every $[x]\in G\backslash \widehat{M}$ the metric $h_\lambda$ on $T_{[x]}(G\backslash \widehat{M})$ restricts to the same metric as $h_0$ on the vertical subspace ${\mathfrak t}_{[x]}=\{Z^*_{[x]};~Z\in{\mathfrak t}\}\subset T_{[x]}(G\backslash \widehat{M}$), because $\lambda$ is $T$-horizontal. Also note that $h_0^T$ and $h_\lambda^T$ coincide on $T\backslash (G\backslash \widehat{M})$. 
\label{notations:one_forms_v}
\end{enumerate}

\end{notations}

The proof of the following theorem is based on \cite[Theorem 1.6]{MR1895349}.

\begin{theorem}

  \label{theorem:isospect}
  Let $(M,g_0)$ be a compact Riemannian manifold and let $G$, $T$ be compact Lie groups acting isometrically on $(M,g_0)$ such that the $G$- and $T$-actions commute. Moreover, assume that $T$ is abelian and acts effectively on $G\backslash M$. Let $\kappa$, $\kappa^\prime$ be two admissible $\mathfrak{t}$-valued $1$-forms on $M$ satisfying: 

For every $\mu\in{\mathcal L}^*$ there is a $G$- and $T$-equivariant $E_\mu\in\Isom(M,g_0)$ such that
\begin{equation}
\label{equation:Fmu}
\mu\circ\kappa=E_\mu^*(\mu\circ\kappa^\prime).
\end{equation}

Then $(G\backslash M,d_{g_\kappa})$ and $(G\backslash M,d_{g_{\kappa^\prime}})$ are isospectral Alexandrov spaces.
\end{theorem}

\begin{proof}
  To apply Theorem \ref{theorem:torusaction} let $W$ be a subtorus of $T$ of codimension $1$ and choose $\mu\in{\mathcal L}^*$ such that $\ker\mu=T_eW$. By assumption, there is $E_\mu\in\Isom(M,g_0)$ satisfying \eqref{equation:Fmu}. We will show that $E_W:=E_\mu$ satisfies the conditions of Theorem \ref{theorem:torusaction}. Since $E_\mu$ is an isometry on $(M,g_0)$, the remarks above imply $E_\mu^*\dvol_{g_\kappa^\prime}=\dvol_{g_\kappa}$. Let $\lambda,\lambda^\prime$ denote the $\mathfrak{t}$-valued $1$-forms on $G\backslash \widehat{M}$ induced by $\kappa$, $\kappa^\prime$, respectively. Recalling from \ref{notations:one_forms}(\ref{notations:one_forms-iii}) that the Riemannian submersion metrics on $G\backslash \widehat{M}$ induced by $g_\kappa$ and $g_{\kappa^\prime}$ are just $h_\lambda$ and $h_{\lambda^\prime}$, respectively, we are left to show  that $E_\mu$ induces an isometry between $(W\backslash (G\backslash \widehat{M}),h_\lambda^W)$ and $(W\backslash (G\backslash \widehat{M}),h_{\lambda^\prime}^W)$.

Let $F_\mu$ be the $T$-equivariant isometry on $(G\backslash \widehat{M},h_0)$ induced by $E_\mu$. Then \eqref{equation:Fmu} implies $\mu\circ\lambda=F_\mu^*(\mu\circ\lambda^\prime)$. Let $[x]\in G\backslash \widehat{M}$, let $V\in T_{[x]}(G\backslash \widehat{M})$ be $W$-horizontal with respect to $h_\lambda$ and set $X:=V+\lambda(V)^*_{[x]}$, $Y:={F_\mu}_*X-\lambda^\prime({F_\mu}_*X)^*_{F_\mu([x])}$. As in the proof of \cite[Theorem 1.6]{MR1895349} (also compare \cite{weilandt10a}) our choice of $\mu$ and the relation $\mu\circ\lambda=F_\mu^*(\mu\circ\lambda^\prime)$ imply that $Y$ is the $W$-horizontal component of ${F_\mu}_*V$ with respect to $h_{\lambda^\prime}$. Since $\|Y\|_{h_{\lambda^\prime}}=\|{F_\mu}_*X\|_{h_0}=\|X\|_{h_0}=\|V\|_{h_\lambda}$, we conclude that $F_\mu$ indeed induces an isometry $(W\backslash (G\backslash \widehat{M}),h_{\lambda}^W)\to(W\backslash (G\backslash \widehat{M}),h_{\lambda^\prime}^W)$.
\end{proof}

%%%%%%%%%%%%%%%%%%%%%%%%%%%%%%% 5. Examples
\section{Examples}\label{sec:examples}
The examples of isospectral bad orbifolds in \cite{weilandt10a} can be seen as an application of Theorem \ref{theorem:isospect}. However, in  this section we use this theorem to construct isospectral Alexandrov spaces which are not isometric to Riemannian orbifolds. 

We give examples of isospectral quotients of spheres (Section \ref{subsec:quotients_spheres}) and isospectral quotients of Stiefel manifolds (Section \ref{subsec:quotients_stiefel}). The former turn out to be a special case of the latter (see Remark \ref{rem:spheres_special_case}). In each construction we first give families of isospectral metrics on a manifold $M$ using ideas from \cite{MR1895349} and \cite{MR2008331} (which have been inspired by related constructions in \cite{gordon2001}) and then observe that taking appropriate $S^1$-quotients gives isospectral families of Alexandrov spaces homeomorpic to $S^1\backslash M$.

% Also note that the isospectral metrics on spheres which come up in the construction of the examples in Section \ref{subsec:quotients_spheres} are very similar to the ones contructed in \cite{gordon2001} -- in fact, they differ only in the choice of the $1$-forms $\kappa$. 3.4(i)

The non-isometry (under certain conditions) will be shown in Section \ref{sec:non_isom} only for the examples in Section \ref{subsec:quotients_spheres}. The fact that our examples are not orbifolds will for all examples follow from the following proposition (\cite[Proposition 6.8]{alexandrino2013}):

\begin{proposition}\label{prop:not_orbifold}
Let $S^1$ act effectively and isometrically on a connected Riemannian manifold $M$. Then $S^1\backslash M$ is isometric to a Riemannian orbifold if and only if the set of points fixed by the whole group $S^1$ is empty or has codimension $2$ in $M$.
\end{proposition}

%%%%%%%%%%%%%%%%%%%%%%%%%%%%%%%%%% Subsection 5.1
\subsection{Quotients of spheres}\label{subsec:quotients_spheres}
Let $m\ge 3$ and denote by $(M,g_0)=(S^{2m+3},g_0)\subset(\C^{m+2},\langle,\rangle)$ the round sphere (with $\langle,\rangle$ the canonical metric). We let $G:=S^1\subset\C$ act on $S^{2m+3}$ via
\[\sigma(u,v):=(\sigma u,v),\]
where $\sigma\in S^1\subset\C$ and $(u,v)\in S^{2m+3}$ with $u\in\C^{m},v\in\C^2$. Moreover, we let the torus $T:=S^1\times S^1\subset\C\times\C$ act on $S^{2m+3}$ via $(\sigma_1,\sigma_2)(u,v_1,v_2):=(u,\sigma_1 v_1,\sigma_2 v_2)$. Both actions are effective and isometric and obviously commute. By $Z_1:=(i,0)$ and $Z_2:=(0,i)$ we will denote the canonical basis of the space $\mathfrak{t}=T_{(1,1)}(S^1\times S^1)\subset\C^2$ which we identify with $\R^2$ via $(it_1,it_2)\mapsto(t_1,t_2)$.

In order to choose appropriate pairs $(\kappa,\kappa^\prime)$ and apply Theorem \ref{theorem:isospect} we recall the following notions from \cite{weilandt10a} (based on \cite{MR1895349}):

\begin{definition}

\label{definition:isomaps}
Let $m\in\N$ and let $j,j^\prime\co {\mathfrak t}\simeq\R^2\to \mathfrak{su}(m)$ be two linear maps.
\begin{enumerate}[(i)]
 \item \label{item:isospec} We call $j$ and $j^\prime$ \emph{isospectral} if for each $Z\in {\mathfrak t}$ there is an $A_Z\in SU(m)$ such that $j_Z^\prime=A_Zj_ZA_Z^{-1}$.
 \item \label{definition:isomaps2} Let $\conj\co\C^m\to\C^m$ denote complex conjugation and set
   \[\mathcal{E}:=\{\phi\in\Aut({\mathfrak t});~\phi(Z_k)\in\{\pm Z_1,\pm Z_2\}\text{ for }k=1,2\}.\]
   We call $j$ and $j^\prime$ \emph{equivalent} if there is $A\in SU(m)\cup SU(m)\circ \conj$ and $\Psi\in\mathcal{E}$ such that $j_Z^\prime=Aj_{\Psi(Z)} A^{-1}$ for all $Z\in {\mathfrak t}$.
 \item \label{definition:isomaps3} We say that $j$ is \emph{generic} if no nonzero element of $\mathfrak{su}(m)$ commutes with both $j_{Z_1}$ and $j_{Z_2}$.
\end{enumerate}

\end{definition}

The following proposition is just a corollary of \cite[Proposition 3.2.6(i)]{MR1895349}. (Note that the definition of equivalence given in \cite{MR1895349} is slightly different from (\ref{definition:isomaps2}) above but this does not affect the validity of the following proposition.)
\begin{proposition}
\label{proposition:isomaps}
For every $m\ge 3$ there is an open interval $I\subset \R$ and a continuous family $j(t)$, $t\in I$, of linear maps $\R^2\to \mathfrak{su}(m)$ such that
\begin{enumerate}

\item[(i)] The maps $j(t)$ are pairwise isospectral.
\item[(ii)] \label{enumitem:nonequiv} For $t_1,t_2\in I$ with $t_1\ne t_2$ the maps $j(t_1)$ and $j(t_2)$ are not equivalent.
\item[(iii)] All maps $j(t)$ are generic.
\end{enumerate}

\end{proposition}

For a linear map $j\co{\mathfrak t}\simeq\R^2\to \mathfrak{su}(m)$ we consider the following $\R^2$-valued 1-form $\kappa$ on $S^{2m+3}$: For $u,U\in\C^{m}$, $v,V\in\C^2$ and $k=1,2$ set
\begin{equation}
  \label{equation:kappa1}
  \kappa_{(u,v)}^k(U,V):=\|u\|^2\langle j_{Z_k}u,U\rangle-\langle U,iu\rangle\langle j_{Z_k}u,iu\rangle
\end{equation}
and pull back to $S^{2m+3}$. Since $\kappa$ does not depend on $v$ or $V$, it is $T$-invariant and $T$-horizontal. Moreover, it is also easily seen to be $S^1$-invariant and $S^1$-horizontal. Now let $j,j^\prime\co \R^2\to\mathfrak{su}(m)$ be isospectral maps and let $\kappa$, $\kappa^\prime$ be the induced $\mathfrak{t}$-valued 1-forms on $S^{2m+3}$. Given an arbitrary element $\mu$ of the dual lattice $\mathcal{L}^\ast$, we set $Z:=\mu(Z_1)Z_1+\mu(Z_2)Z_2$ and choose $A_Z\in SU(m)$ as in Definition \ref{definition:isomaps}(\ref{item:isospec}). Then $E_\mu:=(A_Z,Id_2)\subset SU(m)\times SU(2)$ is obviously $T$-equivariant and satisfies $\mu\circ\kappa=E_\mu^*(\mu\circ\kappa^\prime)$ (see \cite[Proposition 3.2.5]{MR1895349} for a similar calculation). Since $E_\mu$ is also equivariant under our $S^1$-action, Theorem \ref{theorem:isospect} implies that given isospectral $j,j^\prime$, the corresponding Alexandrov spaces $(S^1\backslash S^{2m+3},d_{g_\kappa})$ and $(S^1\backslash S^{2m+3},d_{g_{\kappa^\prime}})$ are isospectral (as we will observe again in Section \ref{subsec:quotients_stiefel}).

Note that the only non-trivial stabilizer is $S^1 $ and the set of fixed points is given by all $[(0,v)]$ with $\|v\|=1$; i.e., for every $\kappa$ the singular part $(S^1\backslash S^{2m+3},d_{g_{\kappa}})^\sing$ is isometric to $S^3$ with the round metric (since $\kappa_{(0,v)}=0$ by \eqref{equation:kappa1}). By Proposition \ref{prop:not_orbifold}, this implies that our quotients are not isometric to Riemannian orbifolds. This observation and the following lemma now show that these examples are genuinely new; i.e., they cannot be obtained using some pair of isospectral orbifolds. We phrase it for isospectral families (with $\kappa(t)$ denoting the admissible form associated with $j(t)$) but of course the same argument also works for pairs. In the following lemma the term Alexandrov space refers to the general Definition \ref{def:as} (without additional assumptions on the regular part).

\begin{lemma}
\label{lem:not_product_sphere}
For every $m\ge 3$ there is a family $j(t)$, $t\in I$, as in Proposition \ref{proposition:isomaps} with the additional property that none of the Alexandrov spaces in the isospectral family $(S^1\backslash S^{2m+3},d_{g_\kappa(t)})$, $t\in I$, is a product of non-trivial Alexandrov spaces.
\end{lemma}

\begin{proof}
Let $j(t)$, $t\in I$, be a familiy of isospectral maps from Proposition \ref{proposition:isomaps}. Scaling all $j(t)$ via the same sufficiently small positive number, the metrics $g_{\kappa(t)}$ can be chosen arbitrarily close to the round metric $g_0$, while the properties from Proposition \ref{proposition:isomaps} are preserved. In particular, we can choose the family $j(t)$, $t\in I$, such that all $(S^{2m+3},g_{\kappa(t)})$ have positive sectional curvature. Then O'Neill's curvature formula (\cite{oneill}) implies that the manifolds $(S^1\backslash S^{2m+3},d_{g_{\kappa(t)}})^\reg$ also have positive sectional curvature.

Suppose that there is $t_0\in I$ and non-trivial Alexandrov spaces $A_1,A_2$ such that $(S^1\backslash S^{2m+3}, d_{g_{\kappa(t_0)}}) = A_1 \times A_2$ (as a product of metric spaces). Since the regular part of a product of Alexandrov spaces is the product of the regular parts of the factors (as follows, e.g., from \cite[Proposition 78]{plaut}), we have $S^3=(S^1\backslash S^{2m+3})^\sing = A_1 \times A_2^\sing \cup A_1^\sing \times A_2$. This observation and the fact that $A_1\times A_2=S^1\backslash S^{2m+3}$ is at least $8$-dimensional imply that either $A_1$ or $A_2$ has no singular points, i.e., is a Riemannian manifold.

Without loss of generality assume that $A_1$ is a manifold and hence $S^3 = A_1 \times A_2^\sing$. Then $A_2^\sing$ (considered as a convex subspace of $S^3$) is an Alexandrov space and hence (since $S^3$ has no singular points) is a Riemannian manifold, too. Hence either $A_1$ or $A_2^\sing$ must be a point.

Since $A_1$ was assumed to be non-trivial, $A_2^\sing$ must be a point and hence $S^3 = A_1$; i.e., $(S^1\backslash S^{2m+3}, d_{g_{\kappa(t_0)}}) = S^3 \times A_2$. But since we have scaled all $j(t)$ in such a way that the connected manifold $(S^1\backslash S^{2m+3}, d_{g_{\kappa(t)}})^\reg$ has positive sectional curvature for all $t \in I$, we conclude that $A_2^\reg$ must be just a point---in contradiction to the fact that $(S^1\backslash S^{2m+3})^\reg$ is at least $8$-dimensional.
\end{proof}

Together with the non-isometry result from Section \ref{sec:non_isom}, the results of this section now prove the theorem stated in the introduction.

%%%%%%%%%%%%%%%%%%%%%%Subsection 5.2: Quotients of Stiefel manifolds
\subsection{Quotients of Stiefel manifolds}\label{subsec:quotients_stiefel}
First we prove a general method to construct isospectral metrics on Alexandrov spaces (Proposition \ref{prop:isospec_construction}) and then apply it to quotients of Stiefel manifolds. For manifolds this method has first been given in \cite{MR2008331}.

Let $H$ and $G$ be compact, connected Lie groups with Lie algebras $\frakh$ and $\frakg$, respectively, and let $T$ be a torus with Lie algebra $\frakt$. We also fix an inner product $\langle,\rangle_\frakt$ on $\frakt$. The following definition generalizes the previous Definition \ref{definition:isomaps}(\ref{item:isospec}):

\begin{definition}[cf. {\cite[Definition 1]{MR2008331}}]
\label{definition:isomapsPrincipal}
Two linear maps $j, j^\prime\co \frakt \to \frakh$ are called \emph{isospectral} if for each $Z \in \frakt$ there exists an $a_Z \in H$ such that $j^\prime_Z = \Ad_{a_Z}(j_Z).$
\end{definition}

Let $(M, g_0)$ be a closed, connected Riemannian manifold, let $G$ and $H \times T$ act isometrically on $(M,g_0)$ and assume that these two actions commute. Given a linear map $j\co \frakt \to \frakh$, each $j_Z := j(Z) \in \frakh$ (with $Z \in \frakt$), induces a vector field $j^\#_Z$ on $M$. We define a $\frakt$-valued $1$-form $\kappa$ on $M$ via
\begin{equation}
  \label{eq:kappa}
  \langle \kappa(X),Z\rangle_\frakt = g_0(X,j_Z^\#)
\end{equation}
for $Z \in \frakt$ and $X \in TM$. Now $\kappa$ is $G$- and $T$-invariant, because the $G$- and $T$-actions both commute with the action of $H$. Moreover, assuming that $\kappa$ is $T$-horizontal and the $T$-orbits meet the $G$-orbits perpendicularly (as will be the case in our examples), we obtain (using the corresponding idea from \cite{MR2008331}) an admissible 1-form $\kappa_{\mathcal H}$ by fixing a basis $(G_1, \ldots, G_l)$ of $\frakg$ and setting

\begin{align*}
  \kappa_{\mathcal H} & (X) := \|G_1^\# \wedge \ldots \wedge G_l^\#\|_0^2 \kappa(X) -\\
  & - \sum_{k=1}^{l} g_0(G_1^\# \wedge \ldots \wedge G_{k-1}^\# \wedge X \wedge G_{k+1}^\# \wedge \ldots \wedge G_l^\#, G_1^\# \wedge \ldots \wedge G_l^\#) \kappa(G_k^\#)
\end{align*}
for $X \in TM$. We will call $\kappa_{\mathcal H}$ the \emph{horizontalization} of $\kappa$.

Applying the ideas from \cite[Proposition 3]{MR2008331} to our situation with an additional $G$-action, we now obtain the following result:

\begin{proposition}\label{prop:isospec_construction}
If $j, j^\prime\co \frakt \to \frakh$ are isospectral linear maps, then the associated $\frakt$-valued $1$-forms $\kappa_{\mathcal H}, \kappa_{\mathcal H}^\prime$ on $M$ satisfy the conditions of Theorem \ref{theorem:isospect}; in particular, $(G\backslash M, d_{g_{\kappa_{\mathcal H}}})$ and $(G\backslash M, d_{g_{\kappa_{\mathcal H}^\prime}})$ are isospectral Alexandrov spaces.
\end{proposition}

\begin{proof}
We check that the condition of Theorem \ref{theorem:isospect} is fulfilled; i.e., for each $\mu \in \mathcal{L}^\ast$ there exists a $T$- and $G$-equivariant $E_\mu \in \Isom(M, g_0)$ such that $\mu \circ \kappa_{\mathcal H} = E_\mu^\ast(\mu \circ \kappa_{\mathcal H}^\prime)$.

With $Z\in\frakt$ defined via $\langle Z,X\rangle_\frakt=\mu(X) ~\forall X\in\frakt$ and $a_Z \in H$ given by Definition \ref{definition:isomapsPrincipal}, we define $E_\mu:=(a_Z , \Id_2)\in H\times T$. Since $H$ commutes with both $G$ and $T$, the isometry $E_\mu$ is $G$- and $T$-equivariant. The equation $\mu \circ \kappa_{\mathcal H} = E_\mu^\ast(\mu \circ \kappa_{\mathcal H}^\prime)$ then follows from $j_Z^\prime = \operatorname{Ad}_{a_Z}(j_Z)$ and the fact that the vector fields $G_k^\#$ are $E_\mu$-invariant (which is the case, because $E_\mu$ is $G$-equivariant).
\end{proof}

\begin{remark}\label{rem:general_horizontalization}
The assumption that $\kappa$ is $T$-horizontal and the $T$-orbits meet the $G$-orbits perpendicularly is not essential. If these conditions were not satisfied, we could use a horizontalization with respect to both the $G$- and the $T$-action and the proposition above would carry over to this more general setting.
\end{remark}

Now we will summarize some basic facts about Stiefel manifolds (of dimensions relevant to our constructions later on). 

As in the previous section let $m\ge 3$. We write $s := m+2$ and fix $r$ such that $1\le r \le s$. By
\[M := V_r(\C^s) = \{Q \in \C^{s \times r};~ Q^\ast Q = I_r\}\]
we denote the complex Stiefel manifold consisting of orthonormal $r$-frames in $\C^s$. $M$ is a manifold of (real) dimension $2r(s-r)$ and the tangent space in $Q\in M$ is given by
\[T_Q M = \{X\in \C^{s \times r};~ Q^\ast X + X^\ast Q = 0\}\]
On $T_Q M$ we use the $U(s)$-left-invariant and $U(r)$-right-invariant Riemannian metric
\begin{equation}
  \label{eq:stiefel-g0}
  {g_0}_{|Q}(X,Y) := \Re \tr(X^\ast Y).
\end{equation}

Note that we do not use the normal homogeneous metric induced by the bi-invariant metric on $U(s)$, since then the $T$-orbits would not meet the $G$-orbits perpendicularly. However, for this metric one could still give a similar construction based on Remark \ref{rem:general_horizontalization}.

Let $H \times T := SU(m) \times (S^1 \times S^1) \subset U(m+2)$ act on the Stiefel manifold $M$ by multiplication from the left and let $G := S^1\simeq U(1)$ be diagonally embedded in $U(m)\simeq U(m)\times\{I_2\}$ and also act via multiplication from the left on $M$. Note that these two effective and isometric actions commute and $T$ acts effectively on $G\backslash M$. As in Section \ref{subsec:quotients_spheres} write $Z_1=(i,0)$, $Z_2=(0,i)$ for the canonical basis of $\mathfrak{t}$. Writing $
\mathfrak{i}:=i$ for the basis of $\mathfrak{g}=i\R$ and

\[\calZ_1 :=\begin{pmatrix}\textbf{0}_m&&\\&i&0\\&0&0\end{pmatrix},~\calZ_2 := \begin{pmatrix}\textbf{0}_m&&\\&0&0\\&0&i\end{pmatrix},~\cali := \begin{pmatrix}iI_m&\\&\textbf{0}_2\end{pmatrix} \in \fraku(s),\]
we have $Z_i^\#(Q) = \calZ_i Q$ and $\fraki^\#(Q) = \cali Q$ and obtain $g_0(Z_i^\#,\fraki^\#)=0$; i.e., the $T$-orbits meet the $G$-orbits perpendicularly. Fixing a linear map $j\co\R^2 \cong \frakt \to \frakh = \fraksu(m)$ and setting
\[\calj_{Z_1} := \begin{pmatrix}j(Z_1) &\\&\textbf{0}_2\end{pmatrix},~\calj_{Z_2} := \begin{pmatrix}j(Z_2) &\\&\textbf{0}_2\end{pmatrix}\in \fraku(s),\]
we obtain $j^\#_{Z_k}(Q) = \calj_{Z_k} Q$ and hence $g_0(Z_i^\#,j^\#_{Z_k})=0$. We conclude that the form $\kappa$ given by \eqref{eq:kappa} is indeed $T$-horizontal and we can consider the associated horizontalisation given by $\kappa_{\mathcal H}(X)=\|\fraki^\#\|_0^2\kappa(X)-g_0(X,\fraki^\#)\kappa(\fraki^\#)$. If we now replace $j$ by an entire isospectral family, Proposition \ref{prop:isospec_construction} (together with Proposition \ref{proposition:isomaps}) implies:

\begin{proposition}\label{prop:isospec_alexandrov}
For every $m \ge 3$ and $1 \le r \le m+2$ there is a continuous family of isospectral maps $j(t)\co\R^2 \cong \frakt \to \frakh = \fraksu(m)$ and each such family defines a continuous family of pairwise isospectral Alexandrov spaces $(S^1\backslash M,d_{\kappa_{\mathcal H}(t)})$.
\end{proposition}

\begin{remark}\label{rem:spheres_special_case}
  Observe that the case $r=1$ gives the examples from Section \ref{subsec:quotients_spheres}: In this case we have $(M,g_0)=(V_1(\C^{m+2}),g_0)=(S^{2m+3},g_0)\subset(\C^{m+2},\langle,\rangle)$ with $g_0=\langle,\rangle$ the round metric. Choosing $\langle,\rangle_{\mathfrak{t}}$ as the standard metric on $\mathfrak{t}\simeq\R^2$, \eqref{eq:kappa} gives $\kappa^k_{(u,v)}(U,V)=\langle U,j_{Z_k}u\rangle$ for $(u,v), (U,V)\in \C^m\times\C^2$. Since $\mathbf{i}^\#_{(u,v)}=(iu,0)$, the horizontalization gives precisely the form which we had called $\kappa$ in \eqref{equation:kappa1}.
\end{remark}

For $r\ge 3$ the action of $G=S^1$ on $M$ is free and $G\backslash M$ is a smooth manifold. For $r = 2$ the only non-trivial stabilizer is $G=S^1$ and
the corresponding fixed points in $M=V_2(\C^{m+2})$ are given by $\left\{\begin{pmatrix}\mathbf{0}_{m\times 2}\\Q\end{pmatrix};~Q\in U(2)\right\}$. Note that since these are also fixed by $H$, every form $\kappa$ given by \eqref{eq:kappa} vanishes on these fixed points and hence the singular part $(G\backslash M,d_{\kappa_{\mathcal H}})^\sing$ is isometric to $U(2)$ with the bi-invariant metric. For $r = 1$ the set of fixed points (which again all have full stabilizer $G$) is a round $3$-sphere by the results in Subsection \ref{subsec:quotients_spheres}. From Proposition \ref{prop:not_orbifold} we conclude that $G\backslash M$ is an orbifold if and only if $r\ge 3$.

Since we already know from Lemma \ref{lem:not_product_sphere} that the examples for the case $r=1$ can be chosen not to be non-trivial products of Alexandrov spaces, we are left to analyze the case $r=2$. The term Alexandrov space now again refers to the general case (Definition \ref{def:as}).

\begin{lemma}\label{lem:not_product_stiefel}
For $r=2$ the Alexandrov spaces from Proposition \ref{prop:isospec_alexandrov} cannot be written as non-trivial products of Alexandrov spaces.
\end{lemma}

\begin{proof}
Suppose that there is $t_0\in I$ such that $(S^1\backslash V_2(\C^{m+2}), d_{\kappa_{\mathcal H}(t_0)}) = A_1 \times A_2$ with $A_1$ and $A_2$ non-trivial Alexandrov spaces. Arguing as in the second and third paragraph of the proof of Lemma \ref{lem:not_product_sphere}, we can without loss of generality assume that $A_1$ is a manifold and obtain $A_1=U(2)$. (Note that the argument from the end of the third paragraph is also applicable here, since $U(2) = (S^1\backslash V_2(\C^{m+2}))^\sing$ endowed with the bi-invariant metric cannot be written as a non-trivial product of Riemannian manifolds: Though it does not have strictly positive sectional curvature, the sectional curvatures of $U(2)$ permit only the decomposition $U(2) = S^1 \times N$, with $S^1$ embedded diagonally in $U(2)$. But then $-Id \in U(2)$ can be reached by a geodesic in $S^1$ and also by a geodesic in $N$, both starting at $Id \in U(2)$---a contradiction.)

Recall that, with $EG\to BG$ denoting the classifying space of $G$, the equivariant cohomology groups are defined for an arbitrary $G$-space $X$ as $H^\ast_G(X) := H^\ast(X \times_G EG)$ and that we have a fibration $X \times_G EG \to BG$ with fiber $X$. In our case this is the fibration $V_2(\C^{m+2}) \times_{S^1} S^\infty \to \CP^\infty$ with fiber $V_2(\C^{m+2})=M$ and so, using the long exact sequence for homotopy groups, we get that the space $V_2(\C^{m+2}) \times_{S^1} S^\infty$ is simply connected, since the fiber and the base are. That means $H^1_{S^1}(V_2(\C^{m+2}); \R) = 0$.

On the other hand (see, e.g., \cite[Example C.8 in the appendix]{equivariant_cohomology_book}) there is an isomorphism $H^1_{S^1}(V_2(\C^{m+2}); \R) \cong H^1(S^1\backslash V_2(\C^{m+2}); \R)$. Since $S^1\backslash V_2(\C^{m+2}) = U(2) \times A_2$ and $H^1(U(2); \R) \cong \R$, we get by the K\"{u}nneth formula the non-triviality of $H^1(S^1\backslash V_2(\C^{m+2}); \R) \cong H^1_{S^1}(V_2(\C^{m+2}); \R)$---a contradiction.
\end{proof}

Propositions \ref{prop:isospec_alexandrov} and \ref{prop:not_orbifold} and Lemma \ref{lem:not_product_stiefel} now show that the examples in this section are isospectral families as announced in the remark in our introduction.

%%%%%%%%%%%%%%%%%%%%%%%%%%%%%%% 6. Non-isometry
\section{Non-isometry}\label{sec:non_isom}

In this section we will show that the examples from Section \ref{subsec:quotients_spheres} are pairwise non-isometric if the maps $j$ and $j^\prime$ are chosen to be non-equivalent and one of them generic (cf. Definition \ref{definition:isomaps}).

\subsection{General theory}
\label{subsection:gen-noniso}
To give a nonisometry criterion for our examples we first recall the notations and remarks from \cite[2.1]{MR1895349}, applied to our special case of the connected $T$-invariant manifold $G\backslash M_\princ$ from Section \ref{section:isometrics}:

\begin{notations}
\label{notations:nonisolist}
\begin{enumerate}[(i)]
\item \label{notations:nonisolisti} A diffeomorphism $F\co G\backslash M_{\text{princ}}\to G\backslash M_{\text{princ}}$ is called \emph{$T$-preserving} if conjugation by $F$ preserves $T\subset\Diffeo(G\backslash M_{\text{princ}})$, i.e., $c^F(z):=F\circ z\circ F^{-1}\in T~\forall z\in T$. In this case we denote by $\Psi_F:=c^F_*$ the automorphism of ${\mathfrak t}=T_eT$ induced by the isomorphism $c^F$ on $T$. Obviously, each $T$-preserving diffeomorphism $F$ of $G\backslash M_{\text{princ}}$ maps $T$-orbits to $T$-orbits; in particular, $F$ preserves $G\backslash \widehat{M}$. Moreover, it is straightforward to show $F_*Z^*={\Psi_F(Z)}^*$ on $G\backslash \widehat{M}$ for all $Z\in\mathfrak{t}$.

\item \label{notations:nonisolist-Aut} Denote by $\Aut^T_{h_0}(G\backslash M_{\text{princ}})$ the group of all $T$-preserving diffeomorphisms of $G\backslash M_{\text{princ}}$ which, in addition, preserve the $h_0$-norm of vectors tangent to the $T$-orbits in $G\backslash \widehat{M}$ and induce an isometry of the Riemannian manifold $(T\backslash (G\backslash \widehat{M}),h_0^T)$. We denote the corresponding group of induced isometries by $\overline{\Aut}^T_{h_0}(G\backslash M_{\text{princ}})\subset \Isom(T\backslash (G\backslash \widehat{M}),h^T_0)$.

\item \label{notations:nonisolist-D} Define ${\mathcal D} := \{\Psi_F;~F\in\Aut^T_{h_0}(G\backslash M_{\text{princ}})\}\subset\Aut({\mathfrak t})$. Note that ${\mathcal D}$ is discrete and each of its elements preserves the lattice ${\mathcal L}=\ker(\exp\co{\mathfrak t}\to T)$.

\item \label{notations:nonisolist-iv} Let $\omega_0\co T(G\backslash \widehat{M})\to {\mathfrak t}$ denote the connection form on the principal $T$-bundle $\pi\co G\backslash \widehat{M}\to T\backslash(G\backslash \widehat{M})$ associated with $h_0$; i.e., the unique connection form such that for each $[x]\in G\backslash \widehat{M}$ the kernel of ${\omega_0}_{|T_{[x]}G\backslash \widehat{M}}$ is the $h_0$-orthogonal complement of the vertical space ${\mathfrak t}_{[x]}=\{Z^*_{[x]};~Z\in{\mathfrak t}\}$ in $T_{[x]}(G\backslash \widehat{M})$. The connection form on $G\backslash \widehat{M}$ associated with $h_\lambda$ is then easily seen to be given by $\omega_\lambda:=\omega_0+\lambda$.

\item Let $\Omega_\lambda$ denote the curvature form on the manifold $T\backslash (G\backslash \widehat{M})$ associated with the connection form $\omega_\lambda$ on $G\backslash \widehat{M}$. Since $G$ is abelian, we have $\pi^*\Omega_\lambda=d\omega_\lambda$.

\item \label{notations:nonisolist-vi} Since $\lambda$ is $T$-invariant and $T$-horizontal, it induces a ${\mathfrak t}$-valued 1-form $\overline{\lambda}$ on $T\backslash (G\backslash \widehat{M})$. Then $\pi^*\Omega_\lambda=d\omega_\lambda=d\omega_0+d\lambda$ implies $\Omega_\lambda=\Omega_0+d\overline{\lambda}$.
\end{enumerate}
\end{notations}

Using these notations, we now have the following criterion for non-isometry:

\begin{proposition}
\label{proposition:nonisometry}
Let $\kappa,\kappa^\prime$ be $\mathfrak{t}$-valued admissible 1-forms on $M$, denote the induced 1-forms on $G\backslash M_\princ$ by $\lambda,\lambda^\prime$ and assume that $\Omega_\lambda,\Omega_{\lambda^\prime}$ have the following two properties:
\begin{itemize}
\item[(N)] \label{proposition:nonisometry:N} $\Omega_\lambda\notin {\mathcal D}\circ {\ol{\Aut}_{g_0}^T}(G\backslash M_\princ)^*\Omega_{\lambda^\prime}$
\item[(G)] \label{proposition:nonisometry:G} No nontrivial 1-parameter group in $\ol{\Aut}_{g_0}^T(G\backslash M_\princ)$ preserves $\Omega_{\lambda^\prime}$.
\end{itemize}
Then the Alexandrov spaces $(G\backslash M,d_{g_\lambda})$ and $(G\backslash M,d_{g_{\lambda^\prime}})$ are not isometric.
\end{proposition}

\begin{proof}
Since the isometry group of a compact Alexandrov space is a compact Lie group (see \cite{galaz12} and the references therein) and the regular part $G\backslash M_\princ$ of $G\backslash M$ is open, the arguments given in \cite[Section 3.2]{weilandt10a} (based on \cite{MR1895349}) carry over almost literally from orbifolds to our Alexandrov spaces.
\end{proof}

\subsection{The non-isometry of quotients of spheres}
We now apply the criteria from Section \ref{subsection:gen-noniso} to our examples from Section \ref{subsec:quotients_spheres} with the arguments actually being similar to the non-isometry proof in \cite[Section 4.3]{weilandt10a}. Recall that we consider the sphere $M=S^{2m+3}\subset\C^{m+2}$ with the round metric $\langle,\rangle$ and that $G=S^1$ acts on the first $m$ components whereas $T=S^1\times S^1$ acts on the last two. 

As in the examples from \cite{weilandt10a} we have
\[\widehat{S^{2m+3}}=\{(u,v)\in S^{2m+3};~u\ne 0\wedge v_1\ne 0\wedge v_2\ne 0\}.\]
Moreover, for $a\in (0,1/\sqrt{2})$ we consider the subset
\[S_a:=\{(u,v)\in S^{2m+3};~|v_1|=|v_2|=a\}\subset \widehat{S^{2m+3}}.\]
For the following proposition recall from Definition \ref{definition:isomaps}(\ref{definition:isomaps2}) that (with $Z_1=(i,0)$, $Z_2=(0,i)$) we had set $\mathcal{E}=\{\phi\in\Aut({\mathfrak t});~\phi(Z_k)\in\{\pm Z_1,\pm Z_2\}\text{ for }k=1,2\}$ Then (with $\mathcal{D}$ defined in Notation \ref{notations:nonisolist}(\ref{notations:nonisolist-D})) we have:

\begin{lemma}
  \[\mathcal{D}\subset\mathcal{E}\]
\end{lemma}
\begin{proof}
  For $[x]=[(u,v)]\in S^1\backslash \widehat{S^{2m+3}}$ we let $R^{[x]}$ denote the embedding
\[T=S^1\times S^1\ni (\sigma_1,\sigma_2)\mapsto(\sigma_1,\sigma_2)[x]\in S^1\backslash \widehat{S^{2m+3}}\]
and calculate $\langle {Z_j^*}_{[x]},{Z_k^*}_{[x]}\rangle = \langle R^{[x]}_*Z_j,R^{[x]}_*Z_k\rangle = \delta_{jk}|v_j|^2$ for $j,k\in\{1,2\}$. In particular, $Z_1^*$ and $Z_2^*$ are orthogonal vector fields on $S^1\backslash \widehat{S^{2m+3}}$ and the area of the orbit $T[x]\subset S^1\backslash \widehat{S^{2m+3}}$ is $4\pi^2|v_1||v_2|$. If $F\in\Aut^T_{g_0}(S^1\backslash S^{2m+3}_{\text{princ}})$, then an argument analogous to the one in \cite[Lemma 4.12]{weilandt10a} shows that for $a\in(0,1/\sqrt{2})$ the quotient $S^1\backslash S_a\subset S^1\backslash \widehat{S^{2m+3}}$ is $F$-invariant (and obviously $T$-invariant).

Observe that for $x\in S_a$ the pull-back of $\langle,\rangle$ via $R^{[x]}$ is $a^2$ times the standard metric on $T=S^1\times S^1$ and hence the flow lines generated by $Z_1^*$ and $Z_2^*$ through $[x]$ give precisely the geodesic loops in $T[x]\subset S^1\backslash S_a$ of length $2\pi\|{Z_1^*}_{[x]}\|=2\pi a$. Since $F$ preserves $S^1\backslash S_a$, the geodesic loops in $TF([x])=F(T[x])$ through $F([x])$ of length $2\pi a$ are given precisely by the flow lines of $Z_1^*$ and $Z_2^*$ through $F([x])$. Since $F\co T[x]\to F(T[x])$ is an isometry, this implies ${F_*}_{[x]}({Z_j^*}_{[x]})\in\{\pm {Z_1^*}_{F([x])},\pm {Z_2^*}_{F([x])}\}$ and hence $\mathcal{D}\subset\mathcal{E}$
\end{proof}

We can now use the lemma above to show the following proposition, which together with Proposition \ref{proposition:nonisometry} gives a criterion for non-isometry. Its proof basically follows the proof of \cite[Proposition 4.14]{weilandt10a}.

\begin{proposition}

\label{proposition:NG}
  Let $j,j^\prime\co\R^2\to\mathfrak{su}(m)$ be two linear maps and let $\kappa,\kappa^\prime$ denote the associated admissible $\mathfrak{t}$-valued 1-forms on $S^{2m+3}$. If $\lambda,\lambda^\prime$ are the $T$-invariant and $T$-horizontal 1-forms on $S^1\backslash \widehat{S^{2m+3}}$ induced by $\kappa,\kappa^\prime$, then we have:
  \begin{enumerate}[(i)]
  \item If $j$ and $j^\prime$ are not equivalent in the sense of Definition \ref{definition:isomaps}(\ref{definition:isomaps2}), then $\Omega_\lambda$ and $\Omega_{\lambda^\prime}$ satisfy condition (N).
  \item If $j^\prime$ is generic in the sense of Definition \ref{definition:isomaps}(\ref{definition:isomaps3}), then $\Omega_{\lambda^\prime}$ has property (G).
  \end{enumerate}

\end{proposition}

\begin{proof}
  Given a $\mathfrak{t}$-valued differential form, we will use the superscripts $j=1, 2$ to denote its components with respect to the ordered basis $\{Z_1,Z_2\}$ of $\mathfrak{t}$. Recalling that $\xi\co \widehat{S^{2m+3}} \to S^1\backslash \widehat{S^{2m+3}}$ denotes the projection, we easily verify $(\xi^*\omega_0^j)_{(u,v)}(U,V)=\frac{\langle V_j,iv_j\rangle}{|v_j|^2}$ for $(u,v)\in \widehat{S^{2m+3}}$ and $(U,V)\in T_{(u,v)}\widehat{S^{2m+3}}\subset \C^{m+2}$. Now choose some arbitrary $a\in(0,1/\sqrt{2})$ and set $L:= S^1\backslash S_a$. With the superscript $L$ denoting the pullback of a differential form on $\widehat{S^{2m+3}}$ to $\xi^{-1}(L)=S_a$, we obtain $(\xi^*\omega_0^j)^L_{(u,v)}(U,V)=\frac{\langle V_j,iv_j\rangle}{a^2}$ for $(u,v)\in S_a$ and $(U,V)\in T_{(u,v)}S_a$ and calculate
  \[d(\xi^*\omega_0^j)^L_{(u,v)}((U,V),(\widetilde{U},\widetilde{V}))=\frac{1}{a^2}(\langle \widetilde{V}_j,iV_j\rangle-\langle V_j,i\widetilde{V}_j\rangle).\]
  Since $V_j, \widetilde{V}_j\in T_{v_j}S^1(a)$ are real multiples of $iv_j$, we conclude $d(\xi^*\omega_0^j)^L=0$ for $j=1, 2$ and therefore $d\omega_0^L=0$ (with the superscript $L$ now referring to the pull-back to $L$). Writing $\Omega_0^L$ for the $\mathfrak{t}$-valued 2-form on $T\backslash L$ induced by the curvature form $\Omega_0$ (see \ref{notations:nonisolist}), we conclude that $\Omega_0^L$ vanishes.

As in \cite{weilandt10a}, we first note that $T\backslash L$ is isometric to $(\CP^{m-1},(1-2a^2)g_{\text{FS}})$ (with $g_{\text{FS}}$ denoting the Fubini-Study metric). Using $\Omega_0^L=0$, the assumption that $\Omega_\lambda,\Omega_{\lambda^\prime}$ do not satisfy (N) then implies (via \ref{notations:nonisolist}(\ref{notations:nonisolist-vi}) and an argument analogous to \cite{weilandt10a}) that there is $\Psi\in\mathcal{D}\subset\mathcal{E}$ and $A\in SU(m)\cup SU(m)\circ \conj$ such that $d\kappa^L = \Psi\circ(A,I_2)^*d{\kappa^\prime}^L$ on $S_a=\xi^{-1}(L)$. The same calculation as in \cite{weilandt10a} finally shows $j_{\Psi(Z)} = A^{-1}j_Z^\prime A$ for all $Z\in\mathfrak{t}$; i.e., $j$ and $j^\prime$ are equivalent. The argument for (ii) is the same as in \cite[Proposition 4.14]{weilandt10a}.
\end{proof}

\noindent\footnotesize{{\bf Acknowledgements}
The first author's work was partly funded by the Studienstiftung des Deutschen Volkes, the graduate program TopMath of the Elite Network of Bavaria and the TopMath Graduate Center of TUM Graduate School at Technische Universit\"{a}t M\"{u}nchen. He is grateful to Bernhard Hanke for the guidance and support during the work on his Bachelor's thesis and his Master's degree.

The second author was partly funded by FAPESP and Funpesquisa-UFSC.

Both authors would like to thank Marcos Alexandrino, Dorothee Sch\"uth and the referee for various helpful suggestions.
}

\bibliographystyle{plain}

\def\cprime{$'$}

\end{document}